\documentclass{amsart}

\usepackage{amsfonts}
\usepackage{amssymb}

\numberwithin{equation}{section}

\theoremstyle{plain}

\newtheorem{theorem}[subsection]{Theorem}
\newtheorem{proposition}[subsection]{Proposition}
\newtheorem{lemma}[subsection]{Lemma}
\newtheorem{corollary}[subsection]{Corollary}
\newtheorem{conjecture}[subsection]{Conjecture}
\newtheorem*{claim}{Claim}

\theoremstyle{definition}

\newtheorem*{example}{Example}

\renewcommand{\leq}{\leqslant}
\renewcommand{\geq}{\geqslant}

\providecommand{\supp}{\mathop{\rm supp}\nolimits}

\providecommand{\Bohr}{\mathop{\rm Bohr}\nolimits}

\providecommand{\LSpec}{\mathop{\rm LSpec}\nolimits}
\providecommand{\Arg}{\mathop{\rm Arg}\nolimits}
\providecommand{\Prog}{\mathop{\rm Prog}\nolimits}

\newcommand{\wh}{\widehat}

\newcommand{\Z}{\mathbb{Z}}
\newcommand{\N}{\mathbb{N}}
\newcommand{\R}{\mathbb{R}}
\newcommand{\C}{\mathbb{C}}

\newcommand{\T}{\mathbb{T}}

\begin{document}

\title{A Fre{\u\i}man-type theorem for locally compact abelian groups}

\author{Tom Sanders}
\address{Department of Pure Mathematics and Mathematical Statistics\\
University of Cambridge\\
Wilberforce Road\\
Cambridge CB3 0WA\\
England } \email{t.sanders@dpmms.cam.ac.uk}

\begin{abstract}
Suppose that $G$ is a locally compact abelian group with a Haar measure $\mu$. The $\delta$-ball $B_{\delta}$ of a continuous translation invariant pseudo-metric is called \emph{$d$-dimensional} if $\mu(B_{2\delta'}) \leq 2^d\mu(B_{\delta'})$ for all $\delta'\in(0,\delta]$. We show that if $A$ is a compact symmetric neighborhood of the identity with $\mu(nA) \leq n^d\mu(A)$ for all $n \geq d\log d$, then $A$ is contained in an $O(d\log^3 d)$-dimensional ball, $B$, of positive radius in some continuous translation invariant pseudo-metric and $\mu(B) \leq \exp(O(d\log d))\mu(A)$.
\end{abstract}

\maketitle

\section{Introduction}

Suppose, as we shall throughout this paper, that $G$ is a locally compact abelian group endowed with a Haar measure $\mu$. Our
interest lies with subsets of $G$ which behave `roughly' like
subgroups; it is instructive to begin with an example.

Consider the group $\R^d$ endowed with Lebesgue measure. For each $\delta \in
(0,\infty)$ we write $B_\delta$ for the ball of radius
$\delta$ in the usual $\ell^\infty$-norm -- these `balls' are `cubes'. Although each ball contains the identity and the inverses of all its elements, it is not a subgroup of $\R^d$ because it is not (additively) closed.
Indeed, addition maps $B_\delta \times B_\delta$ into $B_{2\delta}$
which is $2^d$ times as large as $B_\delta$. If, however, we
introduce an asymmetry in the domain of the additive
operation we can recover a sort of `approximate closure'.

Suppose $\delta$ and $\delta'$ are positive parameters and note that
addition maps $B_\delta \times B_{\delta'}$ into
$B_{\delta+\delta'}$. Now, if $\delta'$ is small compared with
$\delta/d$ then we have
\begin{equation*}
\mu(B_{\delta+\delta'}) = (2(1+\delta'\delta^{-1})\delta)^d =
\mu(B_{\delta})(1+O(d\delta'\delta^{-1})).
\end{equation*}
In this case $B_{\delta+\delta'}$, which contains $B_{\delta}$, is also
not much larger than $B_{\delta}$; with think of addition as being
`approximately closed', in the sense that $B_{\delta}+B_{\delta'} \approx B_\delta$, and the balls $(B_\delta)_{\delta}$
as behaving `roughly' like a subgroup.

Naturally the above observation can be extended to an arbitrary locally compact abelian group, although here the r{\^o}le of norm has to be replaced by that of a continuous translation invariant pseudo-metric\footnote{Recall that a pseudo-metric is simply a metric $\rho$ without the identity condition that $\rho(x,y)=0$ implies $x=y$.}; suppose $\rho$ is such.  We write $B_\delta$ for the usual ball of radius $\delta$ about the identity, that is
\begin{equation*}
B_{\delta}:=\{x \in G: \rho(x,0_G)\leq \delta\}.
\end{equation*}
As before these balls are symmetric compact neighborhoods of the identity, and we should like to recover the same sort of `approximate closure'. This was made possible in $\R^d$ by the presence of a growth condition which does not occur in general; we say that the ball $B_\delta$ is \emph{$d$-dimensional} if
\begin{equation*}
\mu(B_{2\delta'}) \leq 2^d\mu(B_{\delta'}) \textrm{ for all }
\delta' \in (0,\delta].
\end{equation*}
A technical complication arises because it is \emph{not} always true that if $B_\delta$ is $d$-dimensional then $B_\delta + B_{\delta'} \approx B_\delta$ when $\delta'$ is sufficiently small compared with $\delta/d$. However, in \cite{JB}, Bourgain showed that a Vitali covering argument can be used to recover a useful version of this fact on average over a range of values of $\delta$. See \cite[Lemma 4.24]{TCTVHV} for a distillation of this.

The `approximate closure' property is often enough to allow the
transfer of arguments designed for groups to these `approximate
groups'. Moreover, `approximate groups' are far more abundant than genuine groups, which in many cases makes them more useful. We refer the reader to
\cite{JB,BJGSzem,BJGTS2,TS3AP,IDSGen} and \cite{TCTVHV} for examples. It is worth noting that the apparently more general notion of Bourgain system was introduced in \cite{BJGTS2} as a candidate for `approximate groups', but we shall show in \S\ref{sec.examples} that this notion and ours are essentially equivalent.

Suppose now that $\rho$ is a continuous translation invariant pseudo-metric and $B_\delta$ is $d$-dimensional. It follows immediately that
\begin{equation*}
\mu(B_\delta+B_\delta) \leq \mu(B_{2\delta}) \leq 2^d\mu(B_\delta),
\end{equation*}
and it is natural to ask the inverse question: if $A$ is a symmetric compact neighborhood of the identity and $\mu(A+A) \leq 2^d\mu(A)$, then how economically is $A$
contained in a finite dimensional ball of some continuous translation
invariant pseudo-metric? The answer to this question is called
Fre{\u\i}man's theorem and a proof (for discrete abelian
groups) may be found in the paper \cite{BJGIZR} of Green
and Ruzsa.
\begin{theorem}[Weak Fre{\u\i}man's theorem for discrete abelian groups]\label{thm.uk}
Suppose that $G$ is a discrete abelian group and $A \subset G$ is a
finite symmetric neighborhood of the identity with $|A+A| \leq 2^d|A|$.
Then $A$ is contained in a $2^{O(d)}$-dimensional ball, $B$, of positive radius in some translation invariant pseudo-metric and $|B| \leq \exp(2^{O(d)})|A|$.
\end{theorem}
The essential parts of this theorem are the control on the dimension and size of the ball.  It is, of course, easy to define a translation invariant pseudo-metric such that $A$ is contained in the unit ball as follows:

By Ruzsa's covering lemma (Lemma \ref{lem.ruzsacoveringlemma} below) there is a set $T$ with $|T|=2^{O(d)}$ such that $2(A-A) \subset T+ (A-A)$, and one may thus define a translation invariant pseudo-metric via
\begin{equation*}
\rho(x,0_G) := \inf\{l+\sum_{t \in T}{|l_t|}: x=l.a + \sum_{t \in T}{l_t.t} \textrm{ for some } l \in \{0,1\}, a \in A-A\}.
\end{equation*}
This does not, however, have the appropriate dimension and size bounds in general.  Despite this there is some utility in metrics of a similar type and we refer the reader to \cite{IDS} for details.

In actual fact the structure found in \cite{BJGIZR} is more explicit
than `ball of a metric': it is a multi-dimensional coset
progression, which explains the introduction of the qualifier
`weak' in the title of the theorem. In arbitrary locally compact abelian groups multi-dimensional
coset progressions are too restrictive: consider, for example, how one
might contain a short interval in $\T$ in a multi-dimensional coset
progression. In view of this we drop the qualifier `weak' in more general
results.

It is fairly easy to see that (up to the implied constants) this
theorem is best possible, which is a little unfortunate since we
`lose an exponential' in applying the result to a $d$-dimensional
ball: the theorem tells us that this ball is contained in a
$2^{O(d)}$-dimensional ball of a, possibly different, translation invariant pseudo-metric.

Of course, smaller balls in $B_\delta$ have considerably better growth estimates. Indeed, if $n\delta' \leq 2\delta$ then
\begin{equation*}
\mu(nB_{\delta'}) \leq \mu(B_{n\delta'}) \leq (2n)^d\mu(B_{\delta'})=n^{O(d)}\mu(B_{\delta'}).
\end{equation*}
Proving a corresponding inverse theorem is the objective of
this paper.
\begin{theorem}[A Fre{\u\i}man-type theorem for locally compact abelian groups]\label{thm.weakfreiman}
Suppose that $G$ is a locally compact abelian group and $A \subset
G$ is a compact symmetic neighborhood of the identity with $\mu(nA) \leq n^d\mu(A)$ for all $n \geq d\log d$. Then $A$ is contained in an $O(d\log^3 d)$-dimensional ball, $B$, of positive radius in some continuous translation invariant pseudo-metric and
$\mu(B) \leq \exp(O(d\log d))\mu(A)$.
\end{theorem}
As soon as one begins to use the ball provided one typically loses factors exponential in the dimension which renders the size bound largely secondary to the dimension bound. Moreover, it turns out that the dimension bound cannot be significantly improved: return to our example of $\R^d$ and suppose that $A$ is its unit cube. It is immediate that $\mu(nA) \leq n^d\mu(A)$ for all $n$. Now, suppose that $B_\delta$ is a $d'$-dimensional ball containing $A$. Since $B_\delta$ contains $A$, $B_\delta$ has positive measure and it follows by the Brunn-Minkowski theorem that
\begin{equation*}
2^d\mu(B_{\delta}) \leq \mu(B_{\delta}+B_{\delta}) \leq \mu(B_{2\delta}) \leq 2^{d'}\mu(B_{\delta}).
\end{equation*}
We immediately conclude that $d' \geq d$, and so the dimension bound it tight up to logarithmic factors.

A word of justification is in order regarding the name of the theorem. One expects Fre{\u\i}man-type theorems to concern sets with assumptions on $\mu(2A)$ or $\mu(3A)$ and the above theorem takes instead a hypothesis on $\mu(nA)$ for $n$ large. It follows trivially from Ruzsa's covering lemma that one may pass from a bound on $\mu(2A)$ to one on $\mu(nA)$ for $n$ large. However, this results in an exponential loss in the dimension bound on the metric (c.f. Theorem \ref{thm.uk}) and so we feel the above formulation is more natural: we can pass between an explicit description of a low-dimensional ball in a metric and an extrinsic growth condition with only logarithmic losses.

Although the formulation of Theorem \ref{thm.weakfreiman} lends itself to efficient use, it would probably not have been too difficult to adapt existing proofs of Fre{\u\i}man's theorem to yield this result. The real interest of this paper lies in the new method of proof. On a technical level, it turns out to be more convenient to work with not only balls in a pseudo-metric, but also a slightly different definition of `large spectrum' (introduced in \S\ref{sec.fourieranalysis}). On a conceptual level, the main new contribution is in \S\ref{sec.bohrsetsandbourgainsystems} where we show that Bohr sets with highly structured frequency sets actually behave as if they have much lower dimension than the trivial estimate would suggest. The rest of the argument then involves many of the usual ingredients, in particular an easy generalization of Bogolio{\`u}boff's method and an idea of Schoen; the proof is completed in fairly short order.

The paper now splits as follows. In \S\ref{sec.fourieranalysis} we
record our notation; in \S\ref{sec.examples}, which is logically
unnecessary, we consider some examples of finite dimensional balls
to facilitate understanding; in
\S\ref{sec.covering}, we record some covering lemmas;
\S\S\ref{sec.bohrsetsandbourgainsystems}--\ref{sec.proofoftheorem}
contain the proof of the theorem; finally we conclude in
\S\ref{sec.conjectures} with some conjectures and, in particular, relate our result to the so called
Polynomial Fre{\u\i}man-Ruzsa Conjecture.

\section{Notation}\label{sec.fourieranalysis}

The convolution of two functions $f,g \in L^1(\mu)$ is denoted
$f\ast g$ and is defined by
\begin{equation*}
f \ast g(x):=\int{f(x')g(x-x')d\mu(x')} \textrm{ for all } x \in
G.
\end{equation*}
We write $\wh{G}$ for the dual group of $G$, that is the locally
compact abelian group of continuous homomorphisms $\gamma:G \rightarrow S^1$,
where $S^1:=\{z \in \C: |z|=1\}$, and $\nu$ for its Haar measure. The Fourier transform maps $f \in
L^1(\mu)$ to $\wh{f} \in L^\infty(\nu)$ defined by
\begin{equation*}
\wh{f}(\gamma):=\int{f(x)\overline{\gamma(x)}d\mu(x)} \textrm{ for all } \gamma \in \wh{G}.
\end{equation*}

The dual group provides two natural metrics which we shall make considerable use of.
\begin{example}[Bohr sets]
Suppose that $\Gamma$ is a compact neighborhood in $\wh{G}$. For any
$z \in S^1$ we write $\|z\|$ for the quantity $(2\pi)^{-1}|\Arg z |$,
where the argument is taken to have a value lying in $(-\pi,\pi]$.
We define a continuous translation invariant pseudo-metric on $G$ by
\begin{equation*}
\rho(x,y):=\sup\{\|\gamma(x-y)\|: \gamma \in \Gamma\},
\end{equation*}
and write $\Bohr(\Gamma,\delta)$ for the ball of radius
$\delta$ in $\rho$, calling such sets \emph{Bohr sets}; $\Gamma$ is the
\emph{frequency set} of the Bohr set. Note that since $\Gamma$ is a compact neighborhood, $0 < \mu(\Bohr(\Gamma,\delta)) < \infty$ whenever $\delta <1/2$.
\end{example}
\begin{example}[Large spectra]
Suppose that $A$ is a compact neighborhood in $G$.  Recall that for functions $f \in L^2(\mu)$ we have
\begin{equation*}
 \|f\|_{L^2(\mu(A)^{-2}1_A \ast 1_{-A})}^2:=\int_A{\int_A{|f(a-a')|^2dada'}}
\end{equation*}
where the measures in the integral are $\mu_G$ restricted to $A$ and normalized to have total mass $1$. Using this we define a
continuous translation invariant pseudo-metric on $\wh{G}$ by
\begin{equation*}
\rho(\gamma,\gamma'):=\|\gamma-\gamma'\|_{L^2(\mu(A)^{-2}1_A
\ast 1_{-A})}=\|1-\gamma\overline{\gamma'}\|_{L^2(\mu(A)^{-2}1_A
\ast 1_{-A})},
\end{equation*}
and write $\LSpec(A,\delta)$ for the ball of radius $\delta$ in $\rho$,
calling such sets \emph{large spectra}. The true utility of this definition emerges when one notes that
\begin{equation*}
\|1-\gamma\|_{L^2(\mu(A)^{-2}1_A \ast 1_{-A})}^2 = 2(1-\mu(A)^{-2}|\wh{1_A}(\gamma)|^2),
\end{equation*}
and hence
\begin{equation*}
\LSpec(A,\delta)=\{\gamma \in \wh{G}:|\wh{1_A}(\gamma)| \geq \sqrt{1-\delta^2/2}\mu(A)\}.
\end{equation*}
It is, of course, this fact which motivates the name `large spectrum'.
\end{example}

\section{Finite dimensional balls}\label{sec.examples}

There are two standard examples of finite dimensional balls which
it is instructive to consider.
\begin{example}[Multi-dimensional progressions]
Suppose that $T$ is a finite subset of a discerete abelian group $G$.
We define a continuous translation invariant pseudo-metric on $G$ by
\begin{equation*}
\rho(x,y):=\inf\{\sup_{t \in T}{|\sigma_t|}: \sigma \in \Z^T
\textrm{ and } x-y = \sigma.T\},
\end{equation*}
and write $\Prog(T,L)$ for the ball of radius $L$ in $\rho$. It is easy to see that putting
$X_L:=\{\lfloor 3L/2\rfloor t,\lfloor \delta L/2\rfloor
t: t \in T\}$, we get the inclusion
\begin{equation*}
\Prog(T,2L) \subset \Prog(X_L,1) + \Prog(T,L).
\end{equation*}
Since $|\Prog(X_L,1)| \leq 3^{|X_L|}$ we get that $\Prog(T,L)$ is
an $O(|T|)$-dimensional ball.
\end{example}
\begin{example}[Bohr sets]
Suppose that $\Gamma$ is a finite set of characters on a compact abelian group $G$. In \cite[Lemma
4.19]{TCTVHV} it is shown that Bohr sets
are $O(|\Gamma|)$-dimensional balls as follows:

For each $\theta \in \T^\Gamma$ define the set
\begin{equation*}
B_\theta:=\left\{x \in G: \|\gamma(x) \exp(-2\pi i \theta_\gamma)\|
\leq \delta/2\textrm{ for all } \gamma \in \Gamma\right\}.
\end{equation*}
If $B_\theta$ is non-empty let $x_\theta$ be some member. The map $x
\mapsto x - x_\theta$ is an injection from $B_\theta$ into
$\Bohr(\Gamma,\delta)$, so putting $T_\delta:=\{x_\theta: \theta \in
\prod_{\gamma \in
\Gamma}{\{-3\delta/2,-\delta/2,\delta/2,3\delta/2\}}\}$ we have that
\begin{equation*}
\Bohr(\Gamma,2\delta) \subset T_\delta + \Bohr(\Gamma,\delta).
\end{equation*}
Since $|T_\delta| \leq 4^{|\Gamma|}$ we get that $\Bohr(\Gamma,\delta)$ is an
$O(|\Gamma|)$-dimensional ball.
\end{example}
Multi-dimensional progressions and Bohr sets were both brought under
the auspices of so called Bourgain systems in \cite{BJGTS2}. A \emph{$d$-dimensional Bourgain system} $\mathcal{S}$ in $G$ is a collection
$(S_{\delta})_{\delta \in (0,2]}$ of subsets of $G$ obeying the
following axioms:
\begin{enumerate}
\item (Symmetric neighborhood) $S_\delta$ is a compact symmetric neighborhood of the identity for all $\delta \in (0,2]$;
\item (Nesting) $S_{\delta'} \subset S_{\delta}$ for all $\delta,\delta' \in (0,2]$ with $\delta' \leq \delta$;
\item (Subadditivity) $S_{\delta} + S_{\delta'} \subset S_{\delta + \delta'}$ for all $\delta,\delta' \in (0,2]$ with $\delta + \delta' \leq 2$;
\item (Growth) $\mu(S_{2\delta}) \leq 2^d\mu(S_\delta)$ for all $\delta
\in (0,1]$.
\end{enumerate}
Certainly if $B_\delta$ is a $d$-dimensional ball in some continuous translation invariant pseudo-metric $\rho$ then $(B_{\eta\delta})_{\eta \in (0,2]}$ is a $d$-dimensional Bourgain system. For all practical purposes, the Birkhoff group metric
construction from \cite{GB} provides a converse to this. We include
the argument for completeness.
\begin{proposition}
Suppose that $\mathcal{S}=(S_\delta)_{\delta \in (0,2]}$ is a
$d$-dimensional Bourgain system. Then there is a continuous translation invariant pseudo-metric such that
$S_{\delta/2^2} \subset B_\delta \subset
S_{\delta}$ for all $\delta \in (0,2]$. In particular $B_1$ is $O(d)$-dimensional, $S_{2^{-2}} \subset B_1$ and $\mu(B_1) \leq \exp(O(d))\mu(S_{2^{-2}})$. 
\end{proposition}
\begin{proof}
Define the quantities $\rho^*(x,y):=\inf{\{2^{-k}:x-y \in
S_{3^{-k}}\}}$ and
\begin{equation*}
\rho(x,y):=\inf{\{\sum_{k=1}^n{\rho^*(x_{k-1},x_k)}:n \in \N, x_0=x,
x_n=y\}}.
\end{equation*}
The fact that $\rho$ is a continuous translation invariant pseudo-metric is
immediate. For the nesting conclusion it will be sufficient to show
that $\frac{1}{2}\rho^*(x,y) \leq \rho(x,y) \leq \rho^*(x,y)$. The
second inequality is also immediate so it remains to prove the
first.

Suppose we are given $x_0=x, x_1, \dots, x_{n-1}, x_n=y$; write $P:=\rho^*(x_0,x_1) + \dots + \rho^*(x_{n-1},x_n)$ and let $h$ be maximal such that $\rho^*(x_0,x_1) + \dots + \rho^*(x_{h-1},x_h) \leq P/2$. It follows from the maximality that $\rho^*(x_{h+1},x_{h+2}) + \dots + \rho^*(x_{n-1},x_n) \leq P/2$.

Now, by the subadditivity of Bourgain systems, we conclude that
\begin{equation*}
\rho^*(x,x_h) \leq \exp(\log 2 \lfloor \log_3\sum_{k=0}^{h-1}{\rho^*(x_k,x_{k+1})^{\log_2 3}}\rfloor) \leq 2. (P/2),
\end{equation*}
where the second inequality is by nesting of norms. Similarly $\rho^*(x_{h+1},y) \leq P$ and since $\rho^*(x_h,x_{h+1}) \leq P$ is trivially true and $3.S_{3^{-k}}
\subset S_{3^{-(k-1)}}$, it follows that $\rho^*(x,y) \leq 2P$. This
yields the required inequality and the proof is complete.
\end{proof}

\section{Covering lemmas}\label{sec.covering}

Covering lemmas are extremely useful in additive combinatorics and were pioneered by Ruzsa in \cite{IZRArb}.  The most basic is the following which we state for completeness and may be found as \cite[Lemma 2.14]{TCTVHV}.
\begin{lemma}[Ruzsa's covering lemma]\label{lem.ruzsacoveringlemma}
Suppose that $B$ is a compact neighborhood with $\mu(B +
B) \leq 2^k\mu(B)$. Then there is a set $T \subset B$ with $|T|
\leq 2^{O(k)}$ such that $2B-2B \subset T + B-B$.
\end{lemma}
The proof of this result is easy: simply let $T$ be a maximal $B$-separated subset of $2B-2B$.  We leave the details to the reader as we do not require the result. 

There is a refinement of the above lemma due to Chang \cite{MCC} which will be of use to us. The following is a slight reformulation of the result as stated in \cite[Lemma 5.31]{TCTVHV}  so we include a proof for completeness.
\begin{lemma}[Chang's covering lemma]\label{lem.changscoveringlemma}
Suppose that $B$ and $B'$ are compact neighborhoods with $\mu(kB +
B') < 2^k\mu(B')$. Then there is a set $T \subset B$ with $|T|
\leq k$ such that $B \subset \Prog(T,1) + B'-B'$.
\end{lemma}
\begin{proof}
Let $T$ be a maximal \emph{$B'$-dissociated} subset of $B$, that is
a maximal subset of $B$ such that
\begin{equation*}
(\sigma.T + B') \cap (\sigma'.T+B') = \emptyset \textrm{ for all }
\sigma, \sigma' \in \{0,1\}^T.
\end{equation*}
Now suppose that $x' \in B \setminus T$ and write $T':=T \cup
\{x'\}$. By maximality of $T$ there are elements $\sigma,\sigma' \in
\{0,1\}^{T'}$ such that $(\sigma.T' +B') \cap (\sigma'.T' + B') \neq
\emptyset$. Now if $\sigma_{x'} = \sigma'_{x'}$ then $(\sigma|_T.T
+B') \cap (\sigma'|_T.T + B) \neq \emptyset$ contradicting the fact
that $T$ is $B'$-dissociated. Hence, without loss of generality,
$\sigma_{x'}=1$ and $\sigma'_{x'}=0$, whence
\begin{equation*}
x' \in \sigma'|_T.T -\sigma|_T.T + B' - B' \subset \Prog(T,1) + B'-B'.
\end{equation*}
We are done unless $|T|>k$, so let $T' \subset T$ be a set of size
$k$. Denote $\{\sigma.T': \sigma \in \{0,1\}^{T'}\}$ by $S$ and note
that $S \subset kB$ whence
\begin{equation*}
2^k\mu(B') \leq \mu(S+B') \leq \mu(kB+B') < 2^k\mu(B').
\end{equation*}
This contradiction completes the proof.
\end{proof}

\section{Growth of Bohr sets}\label{sec.bohrsetsandbourgainsystems}

When the frequency set of a Bohr set is structured in a particular way there are
better estimates for its growth; \emph{c.f.} \S\ref{sec.examples}.
\begin{proposition}\label{prop.approximateannihilatorsofapproximategroups}
Suppose that $\Gamma$ is a compact symmetric neighborhood of the trivial character
with $\Gamma+\Gamma \subset \Prog(X,1)+\Gamma$ for some finite set
$X$, and $\delta \in (0,2^{-4}]$ is a parameter. Then
\begin{equation*}
\mu(\Bohr(\Gamma \cup X,2\delta)) \leq \exp(O(|X|\log
|X|))\mu(\Bohr(\Gamma \cup X,\delta)).
\end{equation*}
\end{proposition}
We require a preliminary result. Suppose that $\Lambda$ is a set of
characters, $k$ is a positive integer and $\delta \in (0,1]$. By the
triangle inequality it is immediate that $\Bohr(\Lambda,\delta)
\subset \Bohr(k\Lambda,k\delta)$; the following elementary lemma can be used to provide a
partial converse.
\begin{lemma}
Suppose that $t$ is a real number, $k$ is a positive integer, $\delta \in (0,1]$ has $k\delta<1/3$ and\footnote{Here $\langle x \rangle$ denotes the distance from $x$ to the nearest integer.} $\langle rt\rangle \leq k\delta$ for all $r \in \{1,\dots,k\}$. Then $\langle t \rangle \leq \delta$.\hfill $\Box$
\end{lemma}
\begin{corollary}\label{lem.technicallemma}
Suppose that $\Lambda$ is a set of characters containing
the trivial character and $k\delta < 1/3$. Then $\Bohr(k\Lambda,k\delta) \subset \Bohr(\Lambda,\delta)$, and hence $\Bohr(k\Lambda,k\delta) = \Bohr(\Lambda,\delta)$. 
\end{corollary}
\begin{proof}
Since $0_{\wh{G}} \in \Lambda$, we have that $r\lambda \in k\Lambda$
for all $r \in \{1,\dots,k\}$. It follows that if $x \in
\Bohr(k\Lambda,k\delta)$ then
\begin{equation*}
\|\lambda(x)^r\|=\|(r\lambda)(x)\| \leq k \delta \textrm{ for all }
r \in \{1,\dots,k\}.
\end{equation*}
If we define $\theta_x \in (-1,1]$ to be such that
$\lambda(x)=\exp(i \pi \theta_x)$, then we can rewrite the above as
\begin{equation*}
\langle r\theta_x\rangle \leq k\delta \textrm{ for all } r \in
\{1,\dots,k\}.
\end{equation*}
It follows from the preceding lemma that $\langle \theta_x \rangle \leq \delta$ and hence that $x \in \Bohr(\Lambda,\delta)$. The result is proved.
\end{proof}
\begin{proof}[Proof of Proposition
\ref{prop.approximateannihilatorsofapproximategroups}] For each
$\theta \in \T^X$ define the set
\begin{equation*}
B_\theta:=\left\{x \in G: \|\gamma(x) \exp(-2\pi i \theta_\gamma)\|
\leq \delta/2^2|X|\textrm{ for all } \gamma \in X\right\}.
\end{equation*}
Put $I:=\{k\delta/2^2|X|: -2^4|X| \leq k \leq 2^4|X|\}$ and note
that
\begin{equation*}
\Bohr(\Gamma \cup X,2\delta) \subset \bigcup\{B_\theta \cap
\Bohr(\Gamma,2\delta) : \theta \in I^X\}.
\end{equation*}
For each $\theta \in I^X$ let $x_\theta$ be some element of
$B_\theta \cap \Bohr(\Gamma,2\delta)$ (if the set is non-empty); the
map $x \mapsto x - x_\theta$ is an injection from $B_\theta \cap
\Bohr(\Gamma,2\delta)$ into $\Bohr(X,\delta/2|X|)\cap \Bohr(\Gamma,2^2\delta)$. Writing $T$ for the set of all such
$x_\theta$s, we have
\begin{equation*}
\Bohr(\Gamma \cup X,2\delta)  \subset T+\Bohr(\Gamma,2^2\delta)\cap
\Bohr(X,\delta/2|X|)
\end{equation*}
Now, by the triangle inequality, we have
\begin{equation*}
\Bohr(\Gamma,2^2\delta) \cap \Bohr(X,\delta/2|X|) \subset
\Bohr(\Gamma +2^3\Prog(X,1),2^3 \delta),
\end{equation*}
and since the trivial character is in $\Gamma$ and $\Gamma+\Gamma \subset
\Gamma + \Prog(X,1)$ we have $\Gamma+2^3\Prog(X,1) \supset 2^3\Gamma$
and $\Gamma+2^3\Prog(X,1) \supset 2^3\Prog(X,1)$, whence
\begin{equation*}
\Bohr(\Gamma +2^3\Prog(X,1),2^3 \delta) \subset
\Bohr(2^3\Gamma,2^3\delta)\cap\Bohr(2^3\Prog(X,1),2^3\delta).
\end{equation*}
Finally, by Corollary \ref{lem.technicallemma} and the fact that $X
\subset \Prog(X,1)$ we have
\begin{equation*}
\Bohr(2^3\Gamma,2^3\delta)\cap\Bohr(2^3\Prog(X,1),2^3\delta) \subset
\Bohr(\Gamma,\delta) \cap \Bohr(X,\delta)
\end{equation*}
and the result follows on noting that $|T| \leq |I|^{|X|}$.
\end{proof}

\section{Growth of large spectra}\label{sec.growthoflargespectra}

The growth of large spectra is not as neat as that of Bohr sets. Nevertheless, we have the following proposition which leverages a key idea of Schoen \cite{TS} introduced to Fre{\u\i}man-type problems by Green and Ruzsa in \cite{BJGIZR}.
\begin{proposition}\label{prop.LSpecproperties}
Suppose that $A$ is a compact neighborhood with $\mu(nA) \leq n^d\mu(A)$ for all $n\geq d\log d$, and $\epsilon
\in (0,1/2]$ is a parameter. Then
\begin{enumerate}
\item \label{item.LSpecdoubling} either $\epsilon^{-1} = O(d \log^2d)$ or there is a set $X \subset \LSpec(A,2\epsilon)$ with $|X| = O(d \log^2\epsilon^{-1}d)$
such that
\begin{equation*}
\LSpec(A,\epsilon) + \LSpec(A,\epsilon) \subset \Prog(X,1) +
\LSpec(A,\epsilon);
\end{equation*}
\item \label{item.LSpecsize} we have the estimate
\begin{equation*}
\mu(\Bohr(\LSpec(A,\epsilon),1/2\pi)) \leq \exp(O(d\log
\epsilon^{-1}d))\mu(A).
\end{equation*}
\end{enumerate}
\end{proposition}
The proof of the proposition rests on the following claim.
\begin{claim}
For all $\eta \in (0,1/2]$ there is a positive integer $k_{\eta,d}$ with $d\log d \leq k_{\eta,d}=O(\eta^{-2}d\log \eta^{-1}d)$ such that
\begin{equation*}
\int_{\LSpec(A,\eta)}{|\wh{1_A}|^{2k_{\eta,d}}d\nu} \geq
\frac{1}{2}\int{|\wh{1_A}|^{2k_{\eta,d}}d\nu} \geq
\frac{\mu(A)^{2k_{\eta,d}}}{2\mu(k_{\eta,d}A)}.
\end{equation*}
\end{claim}
\begin{proof}
Write $f$ for the $k$-fold convolution of $1_A$ with itself. By
Plancherel's theorem and the Cauchy-Schwarz inequality we have
\begin{equation}\label{eqn.polynomialgrowthCS}
\int{|\wh{1_A}|^{2k}d\nu} = \int{f^2d\mu}
\geq\frac{1}{\mu(\supp f)}\left(\int{f\mu}\right)^2 =
\frac{\mu(A)^{2k}}{\mu(kA)}.
\end{equation}
We split the range of integration on the left into $\LSpec(A,\eta)$
and $\LSpec(A,\eta)^c$:
\begin{eqnarray*}
\int_{ \LSpec(A,\eta)^c}{|\wh{1_A}|^{2k}d\nu} & \leq &
(\sqrt{1-\eta^2/2}\mu(A))^{2k-2}\int{|\wh{1_A}|^2d\nu}\\
& = & (1-\eta^2/2)^{k-1}\mu(A)^{2k-1},
\end{eqnarray*}
by Parseval's theorem.

Now $\mu(kA) \leq k^d\mu(A)$ for $k \geq d\log d$, so there is a positive
integer $k_{\eta,d}$ with $d\log d \leq k_{\eta,d}=O(\eta^{-2}d\log \eta^{-1}d)$ and
\begin{equation*}
(1-\eta^2/2)^{k_{\eta,d}-1} \leq 1/2k_{\eta,d}^d \leq \mu(A)/2\mu(k_{\eta,d}A),
\end{equation*}
whence 
\begin{equation*}
\int_{\LSpec(A,\eta)^c}{|\wh{1_{A}}|^{2k_{\eta,d}}d\nu} \leq
\frac{\mu(A)^{2k_{\eta,d}}}{2\mu(k_{\eta,d}A)},
\end{equation*}
and the claim then follows from the triangle inequality and
(\ref{eqn.polynomialgrowthCS}).
\end{proof}
\begin{proof}[{Proof of Proposition \ref{prop.LSpecproperties},
(\ref{item.LSpecdoubling})}] Since
\begin{equation*}
\int_{\LSpec(A,\eta)}{|\wh{1_A}|^{2k}d\nu} \leq
\nu(\LSpec(A,\eta))\mu(A)^{2k}
\end{equation*}
and
\begin{equation*}
\int{|\wh{1_A}|^{2k}d\nu} \geq
\nu(\LSpec(A,2\eta))(\sqrt{1-2\eta^2}\mu(A))^{2k},
\end{equation*}
we get from the claim that
\begin{eqnarray*}
\nu(\LSpec(A,2\eta)) & \leq &
2(1-2\eta^2)^{k_{\eta,d}}\nu(\LSpec(A,\eta))\\ & = & \exp(O(d\log
\epsilon^{-1}d))\nu(\LSpec(A,\eta))
\end{eqnarray*}
for all $\eta \in (\epsilon/2,1/2]$. Hence, for an integer $r>1$
with $(2r+1/2)\epsilon \leq 1$, we have
\begin{equation*}
\nu(\LSpec(A,(2r+1/2)\epsilon)) \leq \exp(O(d\log r \log
\epsilon^{-1}d))\nu(\LSpec(A,\epsilon/2)).
\end{equation*}
It follows that either $\epsilon^{-1}=O(d\log^2 d)$ or we
may pick $r$ with $r=O(d\log^2 \epsilon^{-1}d)$ such that $(2r+1/2)\epsilon \leq 1$ and
\begin{equation*}
\nu(\LSpec(A,(2r+1/2)\epsilon)) <
2^r\nu(\LSpec(A,\epsilon/2)).
\end{equation*}
Thus, since $\LSpec(A,(2r+1/2)\epsilon) \subset r\LSpec(A,2\epsilon)
+ \LSpec(A,\epsilon/2)$, by Chang's covering lemma (Lemma
\ref{lem.changscoveringlemma}) we have a set $X$ with $|X| \leq r$
such that
\begin{equation*}
\LSpec(A,2\epsilon) \subset \Prog(X,1) + \LSpec(A,\epsilon/2) -
\LSpec(A,\epsilon/2).
\end{equation*}
The result follows.
\end{proof}
\begin{proof}[{Proof of Proposition \ref{prop.LSpecproperties},
(\ref{item.LSpecsize})}] We may assume
$\mu(\Bohr(\LSpec(A,\epsilon),1/2\pi))$ is positive (since otherwise there
is nothing to prove) and hence write $\beta$ for the probability
measure induced on $\Bohr(\LSpec(A,\epsilon),1/2\pi)$ by $\mu$.

Suppose that $\gamma \in \LSpec(A,\epsilon)$. Then, for every $x \in
\Bohr(\LSpec(A,\epsilon),1/2\pi)$ we have
\begin{equation*}
|1-\gamma(x)| = \sqrt{2(1-\cos (\pi \|\gamma(x)\|))} \leq \pi
\|\gamma(x)\| \leq 1/2.
\end{equation*}
Integrating the above calculation with respect to $d\beta$ tells us
that $|1-\wh{\beta}(\gamma)| \leq 1/2$ and it follows by the
triangle inequality that $|\wh{\beta}(\gamma)| \geq 1/2$.
Consequently, by the claim, there is a $k_{\epsilon,d}$ with $d\log d \leq k_{\epsilon,d}=O(\epsilon^{-2} d\log \epsilon^{-1}d)$ such that
\begin{equation*}
\int{|\wh{1_{A}}|^{2k_{\epsilon,d}}|\wh{\beta}|^2d\nu} \geq
2^{-2}\int_{\LSpec(A,\epsilon)}{|\wh{1_{A}}|^{2k_{\epsilon,d}}d\nu}\geq
 \frac{\mu(A)^{2k_{\epsilon,d}}}{2^3\mu(k_{\epsilon,d}A)}.
\end{equation*}
On the other hand
\begin{eqnarray*}
\int{|\wh{1_{A}}|^{2k_{\epsilon,d}}|\wh{\beta}|^2d\nu}
 & \leq & \mu(A)^{2k_{\epsilon,d}-2}\|1_A \ast \beta\|_{L^2(\mu)}^2\\ & \leq & \mu(A)^{2k_{\epsilon,d}-2}\|1_A \ast
 \beta\|_{L^1(\mu)} \|1_A \ast \beta\|_{L^\infty(\mu)}
\end{eqnarray*}
by the Hausdorff-Young inequality, Parseval's theorem and then
H\"{o}lder's inequality. Since $\|1_{A}\ast \beta\|_{L^1(\mu)} =
\mu(A)$ we conclude that
\begin{equation*}
\frac{\mu(A)}{2^3\mu(k_{\epsilon,d}A)} \leq \|1_A \ast
\beta\|_{L^\infty(\mu)} \leq
\frac{\mu(A)}{\mu(\Bohr(\LSpec(A,\epsilon),1/2\pi))}.
\end{equation*}
The result follows since $k_{\epsilon,d} \geq d \log d$.
\end{proof}

\section{Bohr sets with large spectra as frequency sets}

The following lemma describes how Bohr sets and large spectra can be
made to interact. It is only slightly more general than
\cite[Proposition 4.39]{TCTVHV}. The idea of considering the large
spectrum of a sumset was, again, introduced by Green and Ruzsa in \cite{BJGIZR} for the purpose of addressing Fre{\u\i}man-type
problems.
\begin{proposition}\label{prop.lowerbound}
Suppose that $A$ is a compact neighborhood, $l$
is a positive integer such that $\mu(lA) \leq K\mu((l-1)A)$ and
$\epsilon \in (0,1]$ is a parameter. Then
\begin{equation*}
A-A \subset \Bohr(\LSpec(lA,\epsilon),2\epsilon\sqrt{2K}).
\end{equation*}
\end{proposition}
\begin{proof}
Write $\delta = 1-\sqrt{1-\epsilon^2/2}$ and suppose that $\gamma
\in \LSpec(lA,\epsilon)$. Then there is a (real) phase $\omega \in
S^1$ such that
\begin{equation*}
\int{1_{lA}\omega\gamma d\mu} = \omega\wh{1_{lA}}(\gamma) =
|\wh{1_{lA}}(\gamma)| \geq (1-\delta)\mu(lA).
\end{equation*}
It follows that
\begin{equation*}
\int{1_{lA}|1-\omega\gamma|^2d\mu} =2\int{1_{lA}(1-\omega\gamma)
d\mu} \leq 2\delta\mu(lA).
\end{equation*}
If $y_0,y_1 \in A$ then
\begin{equation*}
\int{1_{(l-1)A}|1-\omega\gamma(y_i)\gamma|^2d\mu} \leq
\int{1_{lA}|1-\omega\gamma|^2d\mu} \leq 2\delta\mu(lA).
\end{equation*}
The Cauchy-Schwarz inequality tells us that
\begin{equation*}
|1-\gamma(y_0-y_1)|^2 \leq 2(|1-\omega\gamma(y_0)\gamma(x)|^2
+|1-\omega\gamma(y_1)\gamma(x)|^2)
\end{equation*}
for all $x \in G$, whence
\begin{equation*}
\int{1_{(l-1)A}|1-\gamma(y_0-y_1)|^2d\mu} \leq 2^3\delta\mu(lA).
\end{equation*}
On the other hand
\begin{equation*}
|1-\gamma(x)|^2 = 2(1-\cos(\pi\|\gamma(x)\|)) \geq
2^{-1}\|\gamma(x)\|^2,
\end{equation*}
from which the result follows.
\end{proof}

\section{The proof of the main theorem}\label{sec.proofoftheorem}

We are now in a position to prove our theorem.

\begin{proof}[Proof of Theorem \ref{thm.weakfreiman}]
By the pigeon-hole principle there is some integer $l$ with $d\log d \leq l \leq 2d\log d$ such that $\mu(lA)
\leq 2^{15}\mu((l-1)A)$.

Let $C$ be the absolute constant implicit in the first conclusion,
$\epsilon^{-1}=O(d\log^2 d)$, of Proposition
\ref{prop.LSpecproperties}, (\ref{item.LSpecdoubling}), let $d'=O(d)$ be such that $\mu(n(lA)) \leq n^{d'}\mu(lA)$ for all $n \geq d'\log d'$, and finally let $\epsilon^{-1} := 2^{13}(1+C)d'\log^2 d'$. In view of this choice, by Proposition \ref{prop.LSpecproperties}, (\ref{item.LSpecdoubling})
applied to $lA$ there is some set $X$ with $|X|=O(d\log^2 \epsilon^{-1}d)$ such that
\begin{equation*}
\LSpec(lA,\epsilon)+\LSpec(lA,\epsilon)\subset \Prog(X,1) +
\LSpec(lA,\epsilon).
\end{equation*}
Consider the ball $B=\Bohr(\LSpec(lA,\epsilon)\cup X,2^{-4})$. First, by Proposition \ref{prop.approximateannihilatorsofapproximategroups}, this ball is $O(d\log^3d)$-dimensional. Secondly, since $\LSpec(lA,\epsilon) \cup X
\subset \LSpec(lA,2\epsilon)$ we have
\begin{equation*}
A-A \subset \Bohr(\LSpec(lA,2\epsilon),2^9\epsilon)
\subset \Bohr(\LSpec(lA,2\epsilon),2^{-4})\subset B,
\end{equation*}
by Proposition \ref{prop.lowerbound}. Finally, Proposition
\ref{prop.LSpecproperties}, (\ref{item.LSpecsize}) ensures that
\begin{equation*}
\mu(B) \leq \exp(O(d\log d))\mu(lA) \leq \exp(O(d\log d))\mu(A).
\end{equation*}
\end{proof}

\section{Some concluding conjectures}\label{sec.conjectures}

The following is the Polynomial Fre{\u\i}man-Ruzsa Conjecture for
locally compact abelian groups.
\begin{conjecture}[Polynomial Fre{\u\i}man-Ruzsa
Conjecture] Suppose that $G$ is a
locally compact abelian group and $A \subset G$ is a compact
neighborhood with $\mu(A+A) \leq K \mu(A)$. Then there is a subset $A'$ of $A$ with $\mu(A') \geq
K^{-O(1)}\mu(A)$ contained in a $\log^{O(1)}K$-dimensional ball, $B$ of some continuous translation invariant pseudo-metric and $\mu(B) \leq
\exp(\log^{O(1)}K)\mu(A)$.
\end{conjecture}
In view of Theorem \ref{thm.weakfreiman} the Polynomial Fre{\u\i}man-Ruzsa
Conjecture follows from the next conjecture.
\begin{conjecture}[Polynomial Balog-Szemeredi-Gowers Conjecture]
Suppose that $G$ is a locally compact abelian group and $A \subset
G$ is a compact neighborhood with $\mu(A+A) \leq K\mu(A)$. Then
there is a compact neighborhood $A' \subset A$ with $\mu(A') \geq
K^{-O(1)}\mu(A)$ such that $\mu(nA') \leq n^{\log^{O(1)}K}\mu(A')$ for all $n \geq \log^{O(1)}K$.
\end{conjecture}

In a different direction it is natural, as was done in \cite{BJGTS2}, to conjecture Theorem \ref{thm.weakfreiman} in non-abelian groups.
\begin{conjecture}
Suppose that $G$ is a locally compact group and $A \subset
G$ is a compact neighborhood with $\mu(A^n) \leq n^d\mu(A)$ for all $n \geq d\log d$. Then $A$ is contained in a $d^{1+o(1)}$-dimensional ball, $B$ of some continuous translation invariant pseudo-metric and
$\mu(B) \leq \exp(d^{1+o(1)})\mu(A)$.
\end{conjecture}

\section*{Acknolwedgements}

The author would like to thank Ben Green and the anonymous referee for many useful remarks.

\bibliographystyle{alpha}

\bibliography{master}

\end{document}